\documentclass[12pt]{amsart}
\usepackage{amssymb,latexsym,amscd, amsthm, epsfig,amsmath,amssymb}
\usepackage{color}

 \newtheorem{thm}{Theorem}[section]
 
 \newtheorem{lem}[thm]{Lemma}
 \newtheorem{prop}[thm]{Proposition}
 \theoremstyle{definition}
 \newtheorem{defn}[thm]{Definition}
 \theoremstyle{remark}
 \newtheorem{rem}[thm]{Remark}
 \theoremstyle{definition}

 \newcommand{\PP}{\mathbb{P}}

\usepackage{xypic} 
\usepackage{amsthm}

\begin{document}

\title{Reduced and non-reduced linear spaces: Lines and  points}

\author[E. Carlini]{Enrico Carlini}
\address[E. Carlini]{Dipartimento di Matematica, Politecnico di Torino, Torino, Italia}
\email{enrico.carlini@polito.it}

\author[M.V.Catalisano]{Maria Virginia Catalisano}
\address[M.V.Catalisano]{DIPTEM - Dipartimento di Ingegneria della Produzione, Termoenergetica e Modelli
Matematici, Universit\`{a} di Genova, Piazzale Kennedy, pad. D
16129 Genoa, Italy.} \email{catalisano@diptem.unige.it}

\author[A.V. Geramita]{Anthony V. Geramita}
\address[A.V. Geramita]{Department of Mathematics and Statistics, Queen's University, Kingston, Ontario, Canada, K7L 3N6 and Dipartimento di Matematica, Universit\`{a} di Genova, Genova, Italia}
\email{Anthony.Geramita@gmail.com \\ geramita@dima.unige.it  }


\begin{abstract}
In this paper we consider the problem of determining the Hilbert
function of schemes ${X}\subset\mathbb{P}^n$ which are 
{the generic union of $s$ lines and one $m$-multiple point}. 
We
completely solve this problem for any $s$ and $m$ when $n\geq 4$.
When $n=3$ we find several defective such schemes and conjecture
that they are the only ones. We verify this conjecture in several
cases.
\end{abstract}

\maketitle

\section{Introduction}

If $P$ is a point in $\mathbb{P}^n$ with corresponding ideal
 {
$I_P\subset R= k[x_0,\ldots,x_n]$}
 ($k$ algebraically closed of
characteristic zero), the scheme supported on $P$ and defined by
the ideal 
{
$(I_P)^m$ } is called an $m$-{\it multiple point} with
support $P$. In a remarkable paper \cite{AH95} J. Alexander and A.
Hirschowitz found the Hilbert function of a finite union of
$2$-multiple points supported on a generic set of points in
$\mathbb{P}^n$ (see also \cite{chandler} and \cite{BrOt} for  simpler proofs). This
result permitted Alexander and Hirschowitz to solve the long open
problem regarding the dimensions of the (higher) secant varieties
of the Veronese varieties (see \cite{Ge, IaKa} for an expository
discussion of this important result). In a subsequent paper
\cite{CGG4} the authors showed that, in an analogous way (using
the Lemma of Terracini) one can find the dimensions of the
(higher) secant varieties to Segre embeddings of products of
projective spaces, if one could calculate the Hilbert functions of
certain unions of reduced and non-reduced schemes supported on
unions of generic linear spaces of different dimensions (for more
details see Theorem 1.1. in \cite{CGG4}). The study of such
schemes is one of the principal motivations for our work in this
paper.

There is also other closely related research in the literature,
e.g. some authors have considered the problem of finding the
Hilbert function of generic $m$-multiples points in $\mathbb{P}^2$
(see the survey \cite{Miranda99} and  \cite {CCMO}, \cite
{HarbourneRoe04}, \cite{Yang07}) as well as of  generic $m$-multiple points in
$\mathbb{P}^n$ with $n>2$ (see \cite{LafaceUgaglia06}). Moreover,
Hartshorne and Hirschowitz considered the same problem for a generic
union of (reduced) lines in  $\mathbb{P}^n$ ($n>2$).

In this paper we consider yet another variant of this family of
problems: namely the case in which the scheme
${X}\subset\mathbb{P}^n(n\geq 3)$ is composed of $s$
generic (reduced) lines and one  {generic} $m$-multiple point. A simple
parameter count leads one to expect that the Hilbert function of
such an ${X}$, $HF({X},\cdot)$, is
\[HF({X},d)=\min\left\{\binom{d+n}{n},\binom{m+n-1}{n}+s(d+1)\right\}.\ \ \ \ (*)\]
If we let $hp({X},\cdot)$ denote the Hilbert polynomial of ${X}$, then $(*)$ is really saying that
\[HF({X},d)=\min\left\{hp(\mathbb{P}^n,d),hp({X},d)\right\},\]
equivalently
\[\dim (I_{{X}})_d=\max\left\{\binom{d+n}{n}-\binom{m+n-1}{n}-s(d+a),0\right\}.\]
Note that in this case we say that the Hilbert function of
${X}$ is {\it bipolynomial }(see also \cite{CarCatGer2} for other
examples of this).

We prove $(*)$ (see Theorem \ref{mainThn>3}) for any $s$ and $m$
when $n\geq 4$. When $n=3$, the situation is less clear. In
particular, the ``simple parameter count" no longer always gives
the actual Hilbert function (the precise statement is given in
Theorem \ref{mainThn=3bis}). We conjecture that the parameter
count fails (for $n=3$) if and only if $m=d$ and $1<s\leq d$. In
these cases we show that $\dim (I_{ X})_d=\binom{d-s+2}{2}$.

\section{Basic facts and notation}\label{basicsection}

%
%
%
%
%

Since we will make use of Castelnuovo's inequality several times, we recall it here in a form more suited to our use (for
notation and proof we refer to \cite{AH95}, Section 2).

\begin {defn}  \label{ResiduoTraccia}
If $X, Y$ are closed subschemes of $\mathbb P^n$, we denote by $Res_Y X$
the scheme defined by the ideal $(I_X:I_Y)$ and we call it the
{\it residual scheme} of $X$ with respect to $Y$, we denote by $Tr_Y X \subset Y$
the schematic intersection $X\cap Y$, and call it
the {\it trace} of $X$ on $Y$. 
{We also denote by $X+Y$ the schematic union of $X$ and $Y$}.
\end {defn}

 \begin{lem} \label{Castelnuovo}{\bf (Castelnuovo's inequality):}
Let $d,\delta \in \mathbb N$, $d \geq \delta$, let ${Y} \subseteq \PP ^n$ be a smooth hypersurface of degree $\delta$,
and let $X \subseteq \PP
^n$ be a  closed subscheme. Then
$$
\dim (I_{X, \PP^n})_d  \leq  \dim (I_{ Res_Y X, \PP^n})_{d-\delta}+
\dim (I_{Tr _{Y} X, Y})_d.
$$
\qed
\end{lem}

The following lemma gives a criterion for adding to a scheme $X
\subseteq \Bbb P^ n$ a set of reduced points lying 
 {on a projective variety  $Y$} 
and imposing independent
conditions to forms of a given degree in the ideal of $X$ (see also  \cite[Lemma 2.2]{CarCatGer2}).

 \begin{lem} \label{AggiungerePuntiSuY} Let $d \in \Bbb N$ and let $X  \subseteq \Bbb P^n$ be  a closed subscheme. 
{Let ${Y} \subseteq \Bbb P ^n$ be a closed reduced irreducible subscheme, and
let $P_1,\dots,P_s$ be generic  points on $Y$.}
If $\dim (I_{X })_d =s$ and $\dim (I_{X +Y})_d =0$, then
$
\dim (I_{X+P_1+\cdots+P_s})_d = 0.$
\par
\qed
\end{lem}

\begin{proof}

By induction on $s$. \\
Since  $(I_{X +Y} )_d= (I_X )_d\cap (I_Y)_d=(0) $ and $\dim (I_{X })_d =s>0$,  let  $f \in (I_{X })_d$,  $f \notin (I_{Y})_d$. Therefore there exists $P \in Y$, $P \notin X$ such that $f(P) \neq 0$. It follows that
$\dim (I_{X +P})_d = s-1$ and thus the same holds for  a generic point
{ $P_1 \in Y$.  
}
So we are done in case $s=1$.

Let $s>1$ and let $X' = X+P_1$.
Obviously $\dim (I_{X ' + Y})_d =0$. Hence, by the inductive hypothesis, there exist $s-1$  
{distinct generic  points }
$P_2, \dots,P_{s}$ in $Y$ such that $\dim (I_{X'+P_2+\cdots+P_{s}})_d =\dim (I_{X+P_1+\cdots+P_s})_d
= 0.$

\end{proof}

\begin {defn}\label{conica degenere}
We say that $C$ is a {\it degenerate conic} if  $C$ is the union
of two intersecting lines $L, M.$  In this case we write
$C=L+M$.
\end {defn}

\begin {defn}\label{3dimsundial}
Let $L$ and $M$ be two intersecting lines in $\mathbb P^n$ ($n \geq 3$), let $P = L \cap M$, and let
$T\simeq \Bbb P^{3}$ be a generic linear space containing the
scheme $L+M$.  We call the scheme $L+M+ 2P|_T$ a {\it degenerate conic with an embedded point} or a {\it
3-dimensional  sundial}
(see \cite {HartshorneHirschowitz}, or \cite[definition 2.6 with $m=1$]{CarCatGer2} ).

\end {defn}

\medskip

The following lemma shows that a  3-dimensional  sundial in $\mathbb P^n$ is a degeneration of two generic lines   in $\mathbb P^n$ (see \cite {HartshorneHirschowitz} for the case $n=3$, and
see  \cite[Lemma 2.5]{CarCatGer2} for the proof in a more general  case).

 \begin{lem}  \label{sundial}
 Let $X_1  \subset \Bbb P^n $ ($n \geq 3$) be the disconnected  subscheme consisting of two skew  lines $L_1$ and $M$ (so the linear span of $X_1$ is $<X_1> \simeq \Bbb P^{3} $). Then there exists a flat family  of subschemes $$X_{\lambda}\subset <X_1> \ \ \ \ \ (\lambda \in k )$$
whose generic fiber is the union of two skew lines and
whose special fibre $X_0$ is the union of

\begin{itemize}
\item
the line $M $,

\item a line $L$ which intersects  $M$ in a point $P$,

\item   the scheme $2P|_ { <X_1>}$, that is, the schematic intersection of the double point
$2P$ of $\mathbb P^n$ and $<X_1>$.

\end{itemize}
Moreover, if $H \simeq  \Bbb P^{2}$ is the linear span of $L$
and  $M$, then $Res_H(X_0)$ is given by the (simple) point $P$.
 \end{lem}
\qed

\begin{rem} \label{degenerare2rette}
Since  it is easy to see that  in $\mathbb P^n$ ($n \geq 3$) a  3-dimensional  sundial is also  a degeneration of two intersecting lines  and a simple generic point
{which moves toward the intersection point of the two lines,
} by the lemma above we get that  in $\mathbb P^n$ ($n \geq 3$)
a  degenerate conic with an embedded point can be viewed  either as a degeneration of two generic lines, or as a degeneration of a scheme which is the union of a degenerate conic and a simple generic point.

 \end{rem}
Inasmuch as we have upper semicontinuity of the Hilbert function in a flat family, we will use the remark above several times in what follows.


{Now an easy, but useful 
  Lemma.}

\begin{lem} \label{BastaProvarePers=e,e*}
{ Let $X \subset \mathbb P^n$.}

\begin{itemize} 
\item[(i)]
{If $X= X_1+\dots + X_s $ is the  union of non-intersecting  closed subschemes $X_i$,  
if $X'= X_1+\dots + X_{s'} \subset X,$ where $s' <s$ ,
and if
$HF(X,d)= \sum_{i=1}^s HF(X_i,d) $, then  
$$HF(X',d)= \sum_{i=1}^{s'} HF(X_i,d) .$$
}
\item[(ii)]
{If $X= Y + mP$ is the union of a closed subscheme $Y$ and one $m$-multiple point, if $X'= Y+m'P \subset X,$  where $m' <m$, and if 
$HF(X,d)= HF(Y,d) + \binom{m+n-1}{n}$, then  
$$HF(X',d)= HF(Y,d) + \binom{m'+n-1}{n}.$$
}
\item[(iii)]  
If  $ \dim (I_{X})_d = 0$, then  $ \dim (I_{X''})_d =0$, for any subscheme $X'' \supset X$.
 \par
\end{itemize}
\end{lem} 

\begin{proof} (i)
{$$ HF(X,d) =\sum_{i=1}^{s} HF(X_i,d) = 
\sum_{i=1}^{s'} HF(X_i,d)+\sum_{i=s'+1}^{s} HF(X_i,d)
$$
$$\geq HF(X',d) +\sum_{i=s'+1}^{s} HF(X_i,d) 
\geq HF(X,d) .$$
Hence the  inequalites are equalities, and we get 
 the conclusion.}

{
(ii) Since $HF(X,d)= HF(Y,d) + \binom{m+n-1}{n}$, and $I_X= I_Y \cap I_{mP}$,  from the exact sequence
$$0 \longrightarrow R/ I_Y \cap I_{mP}
 \longrightarrow R/ I_Y \oplus R/ I_{mP} 
\longrightarrow R/( I_Y + I_{mP})\longrightarrow 0,
$$
we get that $ \dim (R/( I_Y + I_{mP}))_d =0$ and 
$ HF(mP,d) = \binom{m+n-1}{n} $. It follows that   $ \dim (R/( I_Y + I_{m'P}))_d =0$ and $ HF(m'P,d) = \binom{m'+n-1}{n} $. Thus from the analogous sequence for $m'P$ we get
$$HF(X',d)= HF(Y,d) + HF(m'P,d) = HF(Y,d) + \binom{m'+n-1}{n} .
$$
}

(iii) Obvious.
\end{proof}

\begin{lem} 
Let $X= X_1+\dots + X_s \subset \mathbb P^n$ be the  union of non intersecting  closed subschemes $X_i$,  let  $s' <s$ and
$$X'= X_1+\dots X_{s'} \subset X.$$

\begin{itemize}
\item[(i)]
If  $\dim (I_{X})_d = {d+n \choose n} - \sum_{i=1}^s HF(X_i,d) $ (the expected value), then also $ \dim (I_{X'})_d $ is as expected, that is
$$\dim (I_{X'})_d = {d+n \choose n} - \sum_{i=1}^{s'} HF(X_i,d) .$$

\item[(ii)]  If  $ \dim (I_{X})_d = 0$, then  $ \dim (I_{X''})_d =0$, for any subscheme $X'' \supset X$ . \par
\end{itemize}
\end{lem} 

We now  recall the basic theorem of Hartshorne and Hirschowitz about the Hilbert function of generic lines.
\begin{thm} \cite[Theorem 0.1]{HartshorneHirschowitz} \label{HH}
Let $n, d \in \mathbb N$.
For $n\geq 3$, the ideal of the scheme $X\subset \Bbb P^n$   consisting of $s$ generic
 lines has the expected dimension, that is,
$$
\dim (I_X)_d = \max \left \{ {d+n \choose n} -s(d+1); 0 \right \},
$$
or equivalently
$$
HF(X,d) = \min \left \{ hp(\PP^n,d)={d+n \choose n}, hp(X,d)=s(d+1)
\right \},
$$
that is, $X$ has bipolynomial Hilbert function.
 \end{thm}
\qed
\medskip

To be more precise  the following equivalent  statement is the actual theorem proved in \cite{HartshorneHirschowitz}:

\begin{thm} \cite[Theorem 0.2]{HartshorneHirschowitz} \label{HH2}
Let $n, d \in \mathbb N$. Let
$$ t = \left\lfloor{d+n \choose n} \over {d+1} \right \rfloor ; \ \ \ \ r = {d+n \choose n}-t(d+1),
$$
and let $L_1, \dots, L_{t+1}$ be $t+1$ generic lines in $\Bbb P^n$.
For $n\geq 3$, the ideal of the scheme $X\subset \Bbb P^n$   consisting of the $t$ lines
$L_1, \dots, L_{t}$  and r generic points lying on $L_{t+1}$
 has the expected dimension, that is,
$$
\dim (I_X)_d = 0 .
$$
 \end{thm}
\qed
\medskip

\begin{rem} \label{equivalenzaEnunciati} By Lemma \ref{BastaProvarePers=e,e*}, the statement of Theorem \ref {HH2} easily implies
  the one of Theorem  \ref{HH}; moreover,  by
Lemma \ref{AggiungerePuntiSuY}, it is easy to prove that also the converse holds.

\end{rem}
\medskip



We now recall the following technical result, we refer the reader
to \cite{3dimsundials} for a proof.

\begin{thm} \label{tantesundials3dim}  Let $n \geq 3$ and let
$X\subset \mathbb P^n $ be the union of  $s$ generic 3-dimensional  sundials and  $l$ generic lines. Then $X$ has bipolynomial Hilbert function, that is,
$$ HF(X,d) = \min \left\{  {d+n \choose n}; (d+1)(2s+l)\right \}
.$$
Equivalently, the following schemes
have the expected Hilbert Function in degree $d$:
$$W =
   \left \{
  \begin{matrix}
  \widehat C_1 + \dots +\widehat  C_{  s }  +P_1+\dots +P_r  & {\hbox  {for } } \  t \    {\hbox  {even } }\\
   \widehat C_1 + \dots +\widehat C_{ s }  +M +P_1+\dots P_r  & {\hbox  {for } } \ t \    {\hbox  {odd } }
   \end{matrix}
    \right. ,
     $$
     $$T =
   \left \{
  \begin{matrix}
  \widehat C_1 + \dots +\widehat  C_{   s }  +M & {\hbox  {for } } \  t \    {\hbox  {even and  }  r>0}\\
   \widehat C_1 + \dots +\widehat C_{ s+1}   & {\hbox  {for } } \ t \    {\hbox  {odd  and } r>0}
   \end{matrix}
    \right. ,
     $$
     where
      $$t=  \left\lfloor{ {d+n \choose n} \over {d+1} } \right\rfloor,
  \ \ \ \ \   r= {d+n \choose n} - t (d+1)
\ \ \ \ \     s = \left\lfloor{ \frac t 2 } \right\rfloor ,$$
where the  $\widehat C_i$ are degenerate conics with an embedded point, that is 3-dimensional sundials, the $P_i$ are generic points and $M$ is a generic line, that is,
$$\dim (I_W)_d =exp\dim (I_W)_d =    {d+n \choose n} - t (d+1)-r=0;
$$
$$\dim (I_T)_d=exp\dim (I_W)_d = \max \left \{   {d+n \choose n} - (t +1)(d+1); 0 \right \}=0.$$

\end{thm}

\section{The main theorem in $\mathbb P^n$, for $n \geq 4$ }\label{sezP4}

In this section
we will prove (see Theorem \ref{mainThn>3}) that for $n\geq 4$, the ideal of the scheme $X\subset \Bbb P^n$   consisting of $s$ generic  lines and a generic point of multiplicity $m$ has the expected dimension.
We start with the following proposition, which, for
$m\leq d$,   is equivalent to Theorem \ref{mainThn>3} (see Remark \ref{equivalenzaEnunciati}  for an  analogous situation).

\begin{prop}  \label{Propn>3}
Let $n, d,m \in \mathbb N$, $n\geq 4$, $m\leq d$. Let

$$e= \left \lfloor   {{{d+n \choose n} -  {m+n-1 \choose n} }\over {d+1} }\right \rfloor
 ; \ \ \ \
r={d+n \choose n} -  {m+n-1 \choose n}  -  e (d+1).
 $$
The ideal of the scheme $X\subset \Bbb P^n$   consisting of $e$ generic
 lines $L_1, \dots, L_e$, $r$ generic points $P_1, \dots, P_r$ lying on a generic line $L$ and a generic point $P$ of multiplicity $m$ has the expected dimension, that is,
$$
\dim (I_X)_d = {d+n \choose n} -  {m+n-1 \choose n}  -  e (d+1) -r =0.
$$

\end{prop}

\begin{proof}

We will prove the theorem by induction on $d-m$.

Let  $d=m$.
Since for  $d=m$ any form of degree $d$ in $I_{ X}$ represents a
cone with $P$ as vertex, it follows that
$$
 \dim (I_X)_d =
  \dim (I_W)_{d} ,
$$
where $W\subset \mathbb P^{n-1}$ consists of $e$ generic lines and $r$ generic points lying on a line. Since for $d=m$ we have
${d+n \choose n} -  {m+n-1 \choose n} = {d+n -1\choose {n-1}}$, we get
$$e=
\left \lfloor   {{{d+n-1 \choose {n-1}} }\over {d+1} }\right \rfloor
 ; \ \ \ \
r={d+n -1\choose {n-1}}  -  e (d+1).
 $$

So by Theorem \ref{HH2} we get
$$
  \dim (I_W)_{d} = {d+n-1 \choose {n-1}} -e(d+1)-r=0 , $$
and we are done for $m=d$.

Assume $m<d$.
Let
$$
e'= \left \lfloor   {{{d-1+n \choose n} -  {m+n-1 \choose n} -r}\over {d} }\right \rfloor
;
$$
$$
r'={d-1+n \choose n} -  {m+n-1 \choose n} -r  -  e'd.
$$
Since $ {{d-1+n \choose n} -  {m+n-1 \choose n} -r}  \geq 0$,
we have $e' \geq 0$  (see the Appendix, Lemma \ref{Propn>3calcoli} (i)).

Notice that  $ e -e'-2r' \geq 0$ (this inequality is treated in the Appendix, Lemma \ref{Propn>3calcoli} (ii)).
Using this inequality we construct a scheme $Y $
obtained from $X$ by specializing some lines and by degenerating other pairs of lines
{into a  hyperplane $H \simeq \Bbb P^{n-1}$.}

More precisely,   we specialize $e-e'-2r'$ lines into $H$ and we degenerate $r'$ pairs of lines in order to obtain  the following specialization of $X$:
$$Y=
 \widehat C_1 + \dots +\widehat C_{r'}  + M_1+\dots+M_{ e-e'-2r' } +L_1 +\dots + L_{e'} +mP+
 P_1+\dots+ P_r,$$
 where the $M_i\subset H$ are  generic lines and  the $ \widehat C_i \subset H_i \simeq \mathbb P^3$  are $3$-dimensional sundials such that      $ \widehat C_i $ is the union of a  degenerate conic $C_i$  lying on $H$ and a double point $2Q_i |_{H_i}\not\subset H$.

So we have
 $$Res_H Y =  Q_1 + \dots +Q_{r'}   +L_1 +\dots + L_{e'} +mP+
 P_1+\dots+ P_r  \subset \Bbb P^{n},$$
$$Tr_H Y  =   C_1 + \dots + C_{r'}  + M_1+\dots+M_{ e-e'-2r' } +T_1 +\dots + T_{e'}  \subset H \simeq  \Bbb P^{n-1},$$
where $T_i= L_i \cap H$ and the $T_i$ are generic points.

Since $e' \geq r'$ (this inequality is proved in the Appendix, Lemma \ref{Propn>3calcoli} (iii)) and $r' \leq d-1$, by Remark \ref{degenerare2rette}, Lemma \ref{BastaProvarePers=e,e*}
and Theorem \ref{tantesundials3dim}, we get that the dimension of  $ \dim  (I_{ Tr_H Y  })_{d} $ is as expected, that is,
$$
 \dim  (I_{ Tr_H Y  })_{d}= {d+n-1 \choose {n-1}} - r'(2d+1) - (e-e'-2r')(d+1)-e'
 $$
 $$= {d+n-1 \choose {n-1}} - (e-e')(d+1)+r'-e'
$$
 $$= {d+n-1 \choose {n-1}} - e(d+1)+{d-1+n \choose n} -  {m+n-1 \choose n} -r
$$
 $$= {d+n-1 \choose {n-1}} - e(d+1)+{d-1+n \choose n}  -
 {d+n \choose n}  + e (d+1)=0.
$$

Now we compute the dimension of the Residue.
Let
$$Res_H Y=Y_1+Y_2,$$
where
$$Y_1= Q_1 + \dots +Q_{r'}   +L_1 +\dots + L_{e'} +mP ,
$$
$$Y_2=  P_1+\dots+ P_r \subset L .
 $$

By the inductive hypothesis , and since $r \leq d$ we have that

$$ \dim  (I_{ Y _1  })_{d-1} ={d-1+n \choose {n}} - r'-e'd- {m+n-1 \choose n}
$$
$$=
{d-1+n \choose {n}}-{d-1+n \choose n} +  {m+n-1 \choose n} +r  + e'd-e'd- {m+n-1 \choose n}
=r,
$$
and
$$ \dim  (I_{ Y _1+L  })_{d-1} =\max \{r-d; 0 \} = 0.$$
Hence, by Lemma \ref{AggiungerePuntiSuY} we get
$$\dim  (I_{ Res_H Y })_{d-1} =0.$$
Now, since $ \dim  (I_{ Tr_H Y  })_{d}=\dim  (I_{ Res_H Y })_{d-1} =0$, by Castelnuovo's Inequality (see Lemma \ref{Castelnuovo}) the conclusion follows.

\end{proof}

\begin{thm}  \label{mainThn>3}
Let $n, d,s,m \in \mathbb N$.
For $n\geq 4$, the ideal of the scheme $X\subset \Bbb P^n$   consisting of $s$ generic
 lines and a generic point $P$ of multiplicity $m$ has the expected dimension, that is,
$$
\dim (I_X)_d = \max \left \{ {d+n \choose n} -  {m+n-1 \choose n} -s(d+1), 0 \right \}.
$$
 \end{thm}

\begin{proof}
Obvious for  $m>d$.
For $m\leq d$  the conclusion follows from Proposition \ref{Propn>3} and  Lemma \ref{BastaProvarePers=e,e*}.

\end{proof}

\section{The main theorem in $\mathbb P^3$ }\label{sezP3}

\begin{prop}  \label{Propn=3}
Let $d \in \mathbb N$,  $d \geq 3$. Let

$$e= \left \lfloor   {{{d+3 \choose 3} - 4}\over {d+1} }\right \rfloor
 ; \ \ \ \
r={d+3 \choose 3} -  4  -  e (d+1).
 $$
The ideal of the scheme $X\subset \Bbb P^3$   consisting of $e$ generic
 lines $L_1, \dots, L_e$, $r$ generic points $P_1, \dots, P_r$ and a generic double point
{supported on  $P$}
  has the expected dimension, that is,
$$
\dim (I_X)_d = {d+3 \choose 3} -  4  -  e (d+1) -r =0.
$$

\end{prop}

\begin{proof}

We will prove the theorem by induction on $d$.
\\

For  $d=3$ we have $e=4$, $r=0$ so
$$X= 2P+L_1+ \cdots +L_4.$$
{Since the trace of $X$ on the plane $<P,L_i>$ is formed by the line $L_i$, one double point and three simple points, then the surfaces defined by the forms of degree 3 in $I_X$ have the plane $<P,L_i>$ as a fixed component.
But the four  planes $<P, L_i>$  cannot be fixed components for  a surface  of degree $3$.}
 It follows that $\dim (I_X)_3 =0$.
\\

For $d=4$  we have $e=6$, $r=1$ so
$$X= 2P+L_1+ \cdots +L_6 + P_1. $$
Now we degenerate the scheme $X$:
first we degenerate the lines $L_1$ and $L_2$, so that they become a 3-dimensional sundial
$ \widehat C $, then we specialize the line $L_3$ on the plane $H=< P, R, P_1 >$,
where$R$ is the   double point  of $ \widehat C$.
Let
$$\widetilde X = 2P+\widehat C  +L_3+ \cdots +L_6 + P_1$$
be the degenerate scheme.

The trace of $\widetilde X$ on the plane $H$ is
$$Tr_H \widetilde X = 2P|H+2R |_H  +L_3+ P_1+ (L_4+L_5+L_6)\cap H \subset H\simeq \mathbb P^2 ,$$
hence
$$\dim (I_ {Tr_H \widetilde X})_4=\dim (I_ { Tr_H \widetilde X -L_3})_3.
$$
Since $(Tr_H \widetilde X -L_3)$ is the union of two double points and four simple points, it follows that
$\dim (I_ {Tr_H \widetilde X})_4 = 0$. So $H$ is a fixed component for the forms of
 $ (I_ { \widetilde X})_4$, and we have
 $$ \dim(I_ { \widetilde X})_4  =  \dim(I_ {Res _H{ \widetilde X}})_3, $$
where ${Res _H{ \widetilde X}}$ is the union of three lines, a point and  a degenerate conic $C$, say
$${Res _H{ \widetilde X}}= P+C+L_4+L_5+L_6.
$$

Now, if we degenerate  $P$ and  $C$, we obtain again the sundial
$\widehat C$, so, by Theorem \ref{tantesundials3dim}  we have
 $$\dim (I_ {  {Res _H{ \widetilde X}}  })_3 = 0, $$
and from here we get  $\dim (I_ X)_4 = 0. $
\\

Now let $d \geq 5$. Let $Q$ be a smooth quadric: we will  specialize some of the lines of  the scheme
$X$ on $Q$.
We consider three cases.
\\
 \par
 {\it Case 1}: $d \equiv 0$ mod 3.  \\
 Let $d = 3h$. Note that, since $d \geq 5$, then $h \geq 2$. We have:
 $$e=   \frac {(h+1)(3h+2)}{ 2} -1,  \ \ \ \ \   r=3(h-1)\geq 3 ,$$
 $$X = 2P +L_1+ ... +L_e +P_1+....+P_r.$$
Let $\widetilde X$ be the scheme obtained from $X$ by  specializing $2h+1$   lines in such a way that the lines $L_1,\dots ,L_{2h+1}$ become  lines of the same ruling on $Q$, (the lines $L_{2h+2},\dots ,L_e$ remain generic lines,  not lying on $Q$), and by specializing on $Q$ the points $P_1$ and $P_2$.  \\
  We have
 $$
 Res_Q  {\widetilde X} =
 2P +L_{2h+2}+ ... +L_e +P_3+....+P_r .$$
 By the inductive hypothesis we have:
 $$\dim (I_{Res_Q  {\widetilde X}} )_{d-2} =
 {3h+1 \choose 3}- 4 - (e-2h-1)(3h-1)- (r-2)
 $$
 $$= {\frac {h(3h+1)(3h-1)} 2}- 4 - {\frac {(h+1)(3h-2)}2}(3h-1)- (3h-5)=0
 .
 $$
 Now
  $$
 Tr_Q  {\widetilde X} =
 L_1+\dots + L_{2h+1}+ Tr_Q(L_{2h+2}+ ... +L_e) +P_1+P_2.$$
 Since the trace on $Q$ of the $(e-2h-1$) lines $L_{2h+2}, \dots,L_e $  consists of $2(e-2h-1)$ generic points, we have that
 $Tr_Q  {\widetilde X} $ consists of $(2h+1)$ lines of the same ruling, and $(2e-4h)$ generic points.
Thinking of $Q$ as $\mathbb P^1 \times \mathbb P^1$, we see that the forms  of degree $3h$ in the ideal of
{$Tr_Q  {\widetilde X} $ }
are  curves of type $(3h-(2h+1), 3h)=(h-1,3h)$ in $\mathbb P^1 \times \mathbb P^1$ passing through
$(2e-4h)$ generic points. Hence
$$\dim ( I_{Tr_Q  {\widetilde X}})_{3h} =h(3h+1)-2e+4h=0
.$$
So by Lemma \ref {Castelnuovo} and by the semicontinuity of the Hilbert function
we get $ \dim (I_{X})_{3h}=   0.$
\\
\par

 {\it Case 2}: $d \equiv 2$ mod 3.  \\
 For computation of this case, recall that we will
 think of $Q$ as $\mathbb P^1 \times \mathbb P^1$ and that (see, for instance, \cite [Section 2] {CGG4}) in the case we are treating each of
 the double points on $Q$ will give three independent condition to our forms.

 Let $d = 3h+2$. We have:
 $$
  \begin{matrix}
{\hbox  {for } } \ h=1: \hfill &  d =5 \hfill ;&    e =8 \hfill ;&  r=4  \hfill   ; \\
 {\hbox  {for } } \ h=2:   \hfill   & d =8 \hfill ;&     e=   17 \ \hfill   ;&  r=8  \hfill    ;\\
 {\hbox  {for } } \ h\geq 3:      & d =3h+2 \   ;&     e=  \frac {3(h+1)(h+2)}{ 2} \  ;&  r=h-3  \  .\\
   \end{matrix}
$$
For $h=1$, we have
$$ X= 2P+L_1+ \cdots +L_8 +P_1+ \cdots+P_4. $$
Specialize the scheme $X$  in such a way that the lines $L_1,\dots ,L_{4}$ become  lines of the same
ruling on $Q$, and the points $P$ and $P_1$ become  points on $Q$. We get
 $$
 Res_Q  {\widetilde X} =  P+L_5+ \cdots +L_8 +P_2+ \cdots+P_4 ,
 $$
{$$
 Tr_Q  {\widetilde X} =  2P|_Q+L_1+ \cdots +L_4 
+ Tr_Q(L_{5}+ ... +L_8) +P_1 ,
 $$}
 and we easily get
 $$\dim (I_{Res_Q  {\widetilde X}} )_{3} = 20-1-16-3=0,
 $$
  $$\dim (I_{Tr_Q  {\widetilde X}} )_{5} = 12-3-8-1=0.
 $$

For $h=2$, we have
$$ X= 2P+L_1+ \cdots +L_{17} +P_1+ \cdots+P_8. $$
Specialize the scheme $X$  so that the lines $L_1,\dots ,L_{6}$ become  lines of the same
ruling on $Q$, and the points $P$, $P_1$ and $P_2$ become  points on $Q$.
We get
 $$
 Res_Q  {\widetilde X} =  P+L_7+ \cdots +L_{17} +P_3+ \cdots+P_8,
 $$
{$$
 Tr_Q  {\widetilde X} =  2P|_Q+L_1+ \cdots +L_6 
+ Tr_Q(L_{7}+ ... +L_{17}) +P_1+P_2 ,
 $$}
 and we have
 $$\dim (I_{Res_Q  {\widetilde X}} )_{6} = 84-1-77-6=0,
 $$
  $$\dim (I_{Tr_Q  {\widetilde X}} )_{8} = 27-3-22-2=0.
 $$

For $h\geq 3$, we have
$$ X= 2P+L_1+ \cdots +L_{e} +P_1+ \cdots+P_{h-3}. $$
Now we degenerate the lines $L_1$ and $L_2$, so that they become a 3-dimensional sundial
$ \widehat C =C+2R$, where C is a degenerate conic and $2R$ is a double point, then we specialize
the points $R$, $P$, $P_1\dots P_{h-3}$ so that they become  points on $Q$, and
the lines $L_3,\dots ,L_{2h+4}$ so that they become  lines of the same
ruling on $Q$.
Let $\widetilde X$ be the specialized scheme.
We have
 $$
 Res_Q  {\widetilde X} =  P+C+L_{2h+5}+ \cdots +L_e,
 $$
 and, by Remark \ref{degenerare2rette}, we get
 $$\dim (I_{Res_Q  {\widetilde X}} )_{3h} ={3h+3 \choose 3} -2(3h+1)-(e-2h-4)(3h+1)=0.
 $$
Moreover
{$$
 Tr_Q  {\widetilde X} 
$$
$$=  2P|_Q+2R|_Q+L_3+ \cdots +L_{2h+4}
+ Tr_Q(L_{2h+5}+ ... +L_e)  +P_1+\cdots+P_{h-3},
 $$}
and  we get
  $$\dim (I_{Tr_Q  {\widetilde X}} )_{3h+2} = (h+1)(3h+3) -3-3-2-2(e-2h-5+1)-(h-3) =0.
 $$
\\
So by Lemma \ref {Castelnuovo} and by the semicontinuity of the Hilbert function
we get $ \dim (I_{X})_{3h+2}=   0.$

 {\it Case 3}: $d \equiv 1$ mod 3.  \\
 Let $d = 3h+1$. Note that $h\geq 2$.  We have:
 $$
  e=  \frac {(h+1)(3h+4)}{ 2} -1 \  ;   \ \ \  r=3h-2   .
$$
Specialize the scheme $X$  in such a way that the lines $L_1,\dots ,L_{2h+1}$ become  lines of the same
ruling on $Q$, and the points $P$ and $P_1, \dots, P_{2h-1}$ become  points on $Q$. Let $\widetilde X$ be the specialized scheme.

So
   $$
 Res_Q  {\widetilde X} =
 P +L_{2h+2}+ ... +L_{e} +P_{2h}+....+P_{3h-2} ,$$
 and by Theorem \ref{HH} we have
   $$\dim (I_{Res_Q  {\widetilde X}} )_{3h-1} =
  {3h+2 \choose 3}- 1-3h(e-2h-1)-(h-1)
 $$
 $$
 = {\frac {h(3h+2)(3h+1)} 2}- 1 - {\frac {9h^2(h+1)}2}-h+1=0
 .
 $$
The trace of $\widetilde X$ on $Q$ consists of the $(2h+1)$ lines of the same ruling $L_1,\dots ,L_{2h+1}$,  the double point $P$,  the simple points
$P_1, \dots, P_{2h-1}$, and the trace of the lines $L_{2h+2}, \dots ,L_{e}$.
As usual,
thinking of $Q$ as $\mathbb P^1 \times \mathbb P^1$, we see that the forms  of degree $3h+1$ in the ideal of
$Tr_Q  {\widetilde X} $ are  curves of type $((3h+1)-(2h+1), 3h+1)=(h,3h+1)$ in $\mathbb P^1 \times \mathbb P^1$. Hence, since it is easy to prove that the double point $P$ gives $3$ independent conditions to our forms
(see, for instance, \cite [Section 2] {CGG4}), we have
$$\dim ( I_{Tr_Q  {\widetilde X}})_{3h+1} =(h+1)(3h+2)-3-(2h-1)-2(e-2h-1)=0.
$$

So also in this case, by Lemma \ref {Castelnuovo} and by the semicontinuity of the Hilbert function,
we get $ \dim (I_{X})_{3h+1}=   0.$
\\

\end{proof}

\begin{thm}  \label{mainThn=3bis}
Let $d,s,m \in \mathbb N$, $d \geq 1$. Let   $X\subset \Bbb P^3$ be the scheme
  consisting of $s \geq 1$ generic  lines and a generic point $P$ of multiplicity $m\geq1$.

 \begin{itemize}
 \item[ (i)] The ideal of  $X\subset \Bbb P^3$   has the expected dimension, that is,

$$
\dim (I_X)_d  = exp \dim (I_X)_d = \max \left \{ {d+3 \choose 3} -  {m+2 \choose 3} -s(d+1), 0 \right \},
$$

(a) for $m>d$;

(b) for $m=d$ and $s>d$, or for $m=d$ and $s=1$;

(c) for $m=d-1$;

(d) for $ m<d-1$ and $1 \leq s \leq m+2$;

(e) for $m= 2$, and $d \geq 3$;

(f)  for $m=1$.
\\

 \item[ (ii)]

For
 $m=d\geq 2$ and $2 \leq s \leq d$,  the dimension of  $ (I_X)_d $ is
$$
\dim (I_X)_d ={d-s+2 \choose 2} \neq exp \dim (I_X)_d ,
$$
and the defect is:
$$\delta = \left \{
\begin{matrix}
 {s \choose 2} &\   \    \hbox   {for } \ &    s \leq {d+2 \over 2}  ; \\
\  &\                \\
 {d-s+2 \choose 2}  &\   \    \hbox   {for } &     {d+2 \over 2}\leq s\leq d     \\
\end{matrix}
\right .
.$$
 \end{itemize}
 \end{thm}

\begin{proof}

(i)  (a) Obvious. We have $ \dim (I_{X})_d =exp \dim (I_{X})_d =0$.

(i) (b) and (ii). If $m=d$   any form of degree $d$ in $I_{ X}$ represents a
cone whose vertex contains $P$.
Hence
$$
\dim (I_X)_d = \dim (I_{X'})_d  , $$
where  $X'\subset \Bbb P^2$ is the projection of $X$ from $P$ in a $\mathbb P^2$ and it is a  scheme
  consisting of $s$ generic  lines. Hence, for $s >d$, we immediately get $\dim (I_X)_d=0$.

  For $s \leq d$ we have
  $$
 \dim (I_{X})_d=
{d-s+2 \choose 2} .$$
Since in this case the expected dimension of $(I_X)_d $ is
$$ exp \dim (I_X)_d = \max \left \{ {d+3 \choose 3} -  {d+2 \choose 3} -s(d+1), 0 \right \}
$$
$$ =\left \{
\begin{matrix}
 {d+2 \choose 2} -s(d+1) & \hbox{ for } & s \leq \frac{d+2}2  \\
 0 & \hbox { for } & s \geq \frac{d+2}2 \\
\end{matrix}
\right.,
$$
then for $s=1$ we have  $\dim (I_{X})_d= exp \dim (I_{X})_d $, and so we are done with
(i)(b).

For $2 \leq s\leq d$ the defect is
$$ \dim (I_{X})_d - exp \dim (I_{X})_d =\left \{
\begin{matrix}
 {s \choose 2}  & \hbox{ for } & s \leq \frac{d+2}2  \\
 \\
 {d-s+2 \choose 2}  & \hbox { for } & s \geq \frac{d+2}2 \\
\end{matrix}
\right. ,
$$
so we  have proved (ii).
\\

(i) (c).  By induction on $d$. Obvious for $d=1$, let $d >1$.

Let
$$X= L_1+ \cdots +L_s+mP
$$
be our scheme, where the $L_i$ are generic lines.
Since $d=m+1$, we have that
$$
exp\dim (I_X)_d = \max \left \{ {d+3 \choose 3} -  {d+1 \choose 3} -s(d+1); 0 \right \}
$$
$$
= \max \left \{ (d+1)^2 -s(d+1); 0 \right \},
$$
hence it is enough to prove that $(I_X)_d$ has the expected dimension    for  $s=d+1$,
{and the conclusion will follows
 from Lemma \ref{BastaProvarePers=e,e*}.
}

Let
$H \simeq \mathbb P^2$ be the plane though $P$ and $L_1$. The trace of $X$ on $H$ is
$$Tr_HX = mP|_H + L_1 + R_2+\cdots +R_{d+1},
$$
where $R_i= L_i \cap H $, and the $R_i$ are $d$   generic points on $H$.

Since  $L_1$ is a fixed component for the curves defined by the forms of $I_{Tr_HX}$, we have
$$\dim (I_{Tr_HX})_d = \dim (I_{Tr_HX -L_1})_{d-1} = {d+1 \choose 2} - {m+1 \choose 2} - d
$$
$$= {d+1 \choose 2} - {d \choose 2} - d=0.$$
It follow that $H$ is a fixed component for the forms of $(I_X)_d$, so
$$\dim (I_X)_d =\dim (I_{Res_HX})_{d-1}
$$
where
$$Res_HX= (m-1)P+L_2+\cdots+L_s = (d-2)P+L_2+\cdots+L_{d+1}.
$$
By the inductive hypothesis we get
$$\dim (I_{Res_HX})_{d-1} = {d+2 \choose 3} - {d \choose 3}- d^2=0,
$$
and we are done with (i) (c).
\\

(i) (d). Since for $m=d-1$, and $s=m+2$ by (i) (c) we have
$$
\dim (I_X)_d = {d+3 \choose 3} -  {m+2 \choose 3} -s(d+1) ,
$$
 by Lemma \ref{BastaProvarePers=e,e*} (i) and (ii) we  get the conclusion.

(i)(e). Let $m= 2$ and $d \geq 3$. We have to prove that

$$
\dim (I_X)_d  = exp \dim (I_X)_d = \max \left \{ {d+3 \choose 3} -  {4} -s(d+1), 0 \right \}.
$$
If $ {d+3 \choose 3} -  {4} -s(d+1)\geq  0$, let
$$e
=\left \lfloor   {{{d+3 \choose 3} -  {4} }\over {d+1} } \right \rfloor
 ; \ \ \
r={d+3 \choose 3} -  4  -  e (d+1),
 $$
and let $P_1, \dots, P_r$ be generic points.

By Proposition \ref{Propn=3} we know that  for $s=e$
$$
\dim (I_{X+P_1+ \cdots +P_r})_d =0,
$$
hence for $s=e$ we have
 $$
\dim (I_{X})_d =r =exp \dim (I_X)_d
$$
and now   the conclusion follows from Lemma \ref{BastaProvarePers=e,e*} (i).

Now let $$ {d+3 \choose 3} -  {4} -s(d+1)< 0 .$$
In this case we have

$$
s >  \frac {{d+3 \choose 3} -  {4} } {(d+1)} =
 \left \{
\begin{matrix}
{ \frac {(h+1)(3h+2)}2}  -  \frac 4 {3h+1}&\   \    \hbox   {for } \ &    d=3h   \\
\  &\                \\
{ \frac {(h+1)(3h+4)}2}  -  \frac 4 {3h+2} &\   \    \hbox   {for } &     d=3h+1     \\
{ \frac {3(h+1)(h+2)}2}  +  \frac {h-3} {3h+3}  &\   \    \hbox   {for } &     d=3h+2     \\
\end{matrix}
\right . 
$$
that is,
$$
s \geq
 \left \{
\begin{matrix}
{ \frac {(h+1)(3h+2)}2}  &\   \    \hbox   {for } \ &    d=3h  ; \\
\  &\                \\
{ \frac {(h+1)(3h+4)}2}  &\   \    \hbox   {for } &     d=3h+1     \\
\\
9   &\   \    \hbox   {for } &     d=5     \\
\\
18   &\   \    \hbox   {for } &     d=8     \\
\\
{ \frac {3(h+1)(h+2)}2}  + 1  &\   \    \hbox   {for } &     d=3h+2   , \ h \geq 3  \\
\end{matrix}
\right . .
$$

Since

$$ t = \left\lceil{d+3 \choose 3} \over {d+1} \right \rceil =
 \left \{
\begin{matrix}
{ \frac {(h+1)(3h+2)}2}  &\   \    \hbox   {for } \ &    d=3h  ; \\
\  &\                \\
{ \frac {(h+1)(3h+4)}2}  &\   \    \hbox   {for } &     d=3h+1     \\
\\
10   &\   \    \hbox   {for } &     d=5     \\
\\
19   &\   \    \hbox   {for } &     d=8     \\
\\
{ \frac {3(h+1)(h+2)}2}  + 1  &\   \    \hbox   {for } &     d=3h+2   , \ h \geq 3  \\
\end{matrix}
\right .
$$
then, except for $d=5$ and $d=8$,
by Theorem \ref{HH} we immediately get  $\dim (I_X)_d =0$.

We remain with the cases $d=5; \ s=9$ and  $d=8; \ s=18$. We omit the proves of these cases.

(i) (f) immediately follows from Theorem \ref{HH}.

\end{proof}

\medskip

\section{Appendix}

\begin{lem}  \label{Propn>3calcoli}  Let
$n \geq4$, $m<d$ and
$$e= \left \lfloor   {{{d+n \choose n} -  {m+n-1 \choose n} }\over {d+1} }\right \rfloor
 ; \ \ \ \
r={d+n \choose n} -  {m+n-1 \choose n}  -  e (d+1) ;
 $$

$$
e'= \left \lfloor   {{{d-1+n \choose n} -  {m+n-1 \choose n} -r}\over {d} }\right \rfloor
;
$$
$$
r'={d-1+n \choose n} -  {m+n-1 \choose n} -r  -  e'd.
$$
Then:
\begin{itemize}

\item[(i)]  $e' \geq 0$;

\item[(ii)]  $e -e'-2r' \geq 0$;

\item[(iii)] $e' \geq r'$.
\end{itemize}
\end{lem}

\begin{proof}
(i)  Since $n\geq 4$ and $r \leq d$, we have
$$e'={d-1+n \choose n} -  {m+n-1 \choose n} -r \geq
{d-1+n \choose n} -  {d-1+n-1 \choose n} -d $$
$$={d+n-2 \choose {d-1}} -d \geq
{d+2 \choose {d-1}} -d \geq 0.$$

(ii)  Since $e'+2r' $ is an integer, then  the inequality
$e \geq  e'+2r'  $ is equivalent to
 ${{{d+n \choose n} -  {m+n-1 \choose n} }\over {d+1} }  \geq e'+2r'  $.
Hence, if we prove that
 $${{d+n \choose n} -  {m+n-1 \choose n}  } -(d+1)e'-2(d+1)r'  \geq 0
 $$
 we are done.

 Now
 $${{d+n \choose n} -  {m+n-1 \choose n}  } -(d+1)e'-2(d+1)r'
$$
$$={{d+n \choose n} +(2d+1) {m+n-1 \choose n}  } +
$$
$$+(d+1)(2d-1)e'-2{d-1+n \choose n} (d+1) +2r (d+1)
$$
$$\geq {{d+n \choose n} +(2d+1) {m+n-1 \choose n}  } +
$$
$$+(d+1)(2d-1)
\left (  {{{d-1+n \choose n} -  {m+n-1 \choose n} -r}\over {d} } -1 \right)
-2{d-1+n \choose n} (d+1) +2r (d+1)
$$
$$= {1 \over d} \left (
d{d+n \choose n} -  (d+1) {d-1+n \choose n}+   {m+n-1 \choose n} +r (d+1) -d(2d^2+d-1)
\right)
$$
$$= {1 \over d} \left (
(n-1){d+n-1 \choose {d-1}} +   {m+n-1 \choose n} +r (d+1) -d(2d^2+d-1)
\right)
$$
$$\geq {1 \over d} \left (
(n-1){d+n-1 \choose {d-1}} -d(2d^2+d-1)
\right).
$$
For $n\geq 5$ we have
$$(n-1){d+n-1 \choose {d-1}} -d(2d^2+d-1) \geq
{1\over {30}} d(d+1)((d+2)(d+3)(d+4)-60d+30) \geq 0
$$
for any $d \geq 1$.

For $n = 4$, we have
$$
(n-1){d+n-1 \choose {d-1}} -d(2d^2+d-1) = {1\over 8} d (d+1) (d^2-11d+14) ,
$$
and this is positive for $d=1$ and $d \geq 10$.
Hence, except  for $n=4$ and $ 2 \leq d \leq 9$, we have proved that $e-e'-2r' \geq 0$.
For $n=4$ by direct computation we find:

$\begin{matrix}
d      &   \ m  &  \  e & \  e'     &  \  r'  &  \    e-e'-2r'  \\
2      &  \  1  &  \  4 & \   1   &     \  0  &  \    3  \\
3      & \   1  &  \  8 & \   4   &    \  0  &  \    4  \\
3      & \   2  &  \  7 & \   2   &    \  2  &  \    1  \\
4      & \   1  &  \  13 & \   7   &    \  2  &  \    2  \\
4      & \   2  &  \ 13 & \   7   &    \  2  &  \    2  \\
4      & \   3  &  \  11 & \   5   &    \  0  &  \    6  \\
5      & \   1  &  \  20 & \   12  &    \  4  &  \    0  \\
5      & \   2  &  \  20 & \   12  &    \  4  &  \    0  \\
5      & \   3  &  \  18 & \   10   &    \  2  &  \    4  \\
5      & \   4  &  \  15 & \   6   &    \  4  &  \    1  \\
6      & \   1  &  \  29 & \   19  &    \  5  &  \    0  \\
6      & \   2  &  \  29 & \   19  &    \  5  &  \    0  \\
6      & \   3  &  \  27 & \   17  &    \  3  &  \    4  \\
6      & \   4  &  \  25 & \   15  &    \  1  &  \    8  \\
6      & \   5  &  \  20 & \   9  &    \  2  &  \    7  \\
7      & \   1  &  \  41 & \   29  &    \  5  &  \    2  \\
7      & \   2  &  \  40 & \   28  &    \  4  &  \    4  \\
7      & \   3  &  \  39 & \   27  &    \  3  &  \    6  \\
7      & \   4  &  \  36 & \   24  &    \  0  &  \    12  \\
7      & \   5  &  \  32 & \   19  &    \  3  &  \    7  \\
7      & \   6  &  \  25 & \   11  &    \  3  &  \    8  \\
8      & \   1  &  \  54 & \   40  &    \  1  &  \    12  \\
8      & \   2  &   \  54 & \   40  &    \  1  &  \    12  \\
8      & \   3  &  \  53 & \   39  &    \  0  &  \    14  \\
8      & \   4  &  \  51 & \   36  &    \  6  &  \    3  \\
8      & \   5  &  \  47 & \   32  &    \  2  &  \    11  \\
8      & \   6  &  \  41 & \   25  &    \  4  &  \    8  \\
8      & \   7  &  \  31 & \   14  &    \  2  &  \    13  \\
9      & \   1  &  \   71 & \   54  &    \  4  &  \    9  \\
9      & \   2  &  \  71 & \   54  &    \  4  &  \    9  \\
9      & \   3  &  \  70 & \   53  &    \  3  &  \    11  \\
9      & \   4  &  \  68 & \   51  &    \  1  &  \    15  \\
9      & \   5  &  \  64 & \   46  &    \  6  &  \    6  \\
9      & \   6  &  \  58 & \   40  &    \  0  &  \    18  \\
9      & \   7  &  \  50 & \   31  &    \  1  &  \    17  \\
9      & \   8  &  \  38 & \   17  &    \  7  &  \    7  \\

\end{matrix}
$

It follows that also in these cases we have $e-e'-2r' \geq 0$, and this completes the proof.

(iii)  $r'$ is an integer, hence it sufficies to prove that

$$  {{{d-1+n \choose n} -  {m+n-1 \choose n} -r}\over {d} } - r' \geq 0.
$$
Since $m \leq d-1$, $r \leq d$ and $r '\leq d-1$, $n \geq 4$ we have
$$
 {{{d-1+n \choose n} -  {m+n-1 \choose n} -r} } - r' d
 $$
 $$
 \geq {d-1+n \choose n} -  {d-1+n-1 \choose n} -d- (d-1)d
$$
$$
= {d+n-2 \choose {d-1}}  -d^2 \geq  {d+2 \choose {3}}  -d^2 = {d \choose {3}}\geq 0,
$$
and the conclusion follows.

\end{proof}


\def\cprime{$'$}

\end{document}